\documentclass[reqno]{amsart}
\usepackage{amsthm, amscd, amsfonts,graphicx}
\usepackage{color}
\usepackage{enumerate}
\usepackage{verbatim}
\numberwithin{equation}{section}
\theoremstyle{plain}
 \newtheorem{theorem}{Theorem}[section]
 \newtheorem{lemma}[theorem]{Lemma}

\theoremstyle{definition}
 \newtheorem{definition}[theorem]{Definition}
 \newtheorem{remark}[theorem]{Remark}

\theoremstyle{remark}

\newenvironment{enumeratei}{\begin{enumerate}[\quad\upshape (i)]} {\end{enumerate}}
\newcommand \dom [1] {\textup{Dom}(#1)}
\newcommand \many{\boldsymbol\sigma}
\newcommand \smany{$\many$-many}
\newcommand \krlist {\mathcal L_{\textup{KR}}}
\newcommand \dual [1] {#1^{\boldsymbol\delta}}

\DeclareMathOperator \Sub {Sub}
\newcommand \leqs {\mathrel{\leq_{S}}}
\newcommand \leql {\mathrel{\leq_{L}}}
\newcommand \poset[1]{(#1;\leq)}
\newcommand \slat[1]{(#1;\vee)}
\newcommand \silat[2]{(#1;\vee_{#2})}
\newcommand \lattice[1]{(#1;\vee,\wedge)}
\newcommand \ofideal[1] {\mathord{\downarrow}_{#1}}
\newcommand \offilter[1] {\mathord{\uparrow}_{\kern-1pt #1}\kern 1pt}

\newcommand \Mi {M}
 
\newcommand \ubin[1] {U_{#1}}
\newcommand \nnul{\mathbb N_0}

\newcommand \then {\Rightarrow}

\newcommand\set [1]{\{#1\}}

\renewcommand \phi {\varphi}
\newcommand \url [1] {{\texttt{#1}}}

\newcommand \red [1] {{\color{red}#1\color{black}}}

\newcommand \nothing [1] {}

\newcommand \plunul [1] {#1^{\cup0}}

\begin{document}
\title[One hundred twenty-seven subsemilattices]
{One hundred twenty-seven subsemilattices and planarity}

\author[G.\ Cz\'edli]{G\'abor Cz\'edli}
\address{University of Szeged, Bolyai Institute, Szeged,
Aradi v\'ertan\'uk tere 1, Hungary 6720}
\email{czedli@math.u-szeged.hu}
\urladdr{http://www.math.u-szeged.hu/~czedli/}

\thanks{This research of was supported by the Hungarian Research, Development and Innovation Office under grant number KH 126581.\hskip5cm\red{July 1, 2019}}

\subjclass{06A12, 06B99, 20M10}

\keywords{Planar semilattice, planar lattice, subsemilattice, number of subalgebras, number of subuniverses, computer-assisted proof, finite semilattice}

\dedicatory{Dedicated to the memory of Ivo G.\ Rosenberg}

\begin{abstract} Let $L$ be a finite $n$-element semilattice. We prove that if $L$ has at least $127\cdot 2^{n-8}$ subsemilattices, then $L$ is planar. For $n>8$, this result is sharp since  there is a non-planar semilattice with exactly $127\cdot 2^{n-8}-1$ subsemilattices.
\end{abstract}

\maketitle

\section{Our result  and introduction} 
This paper is dedicated to the memory of \emph{Ivo G.\ Rosenberg} (1934--2018). In addition to his celebrated theorem describing the maximal clones of operations on a finite set,  
he has published many important results in many fields of mathematics. According to MathSciNet, thirteen of his papers belong to the category ``Order, lattices, ordered algebraic structures''; \cite{ChCzR} is one of these thirteen and it is a great privilege to me that I was included.

In the present paper,  \emph{semilattices} (without adjectives) are understood as \emph{join-semi\-lattices}. In spite of this convention, sometimes we write ``join-semilattice'' for emphasis. Note that Theorem~\ref{thmmain} below is valid also for  commutative idempotent semigroups, because they are, in a well-known sense, equivalent to  join-semilattices. A finite semilattice is said to be \emph{planar} if it has a Hasse diagram that is also a planar representation of a graph. Our goal is to prove that finite semilattices with many subsemilattices are planar. Namely, we are going to prove the following theorem.

\begin{theorem}\label{thmmain} Let $L$ be a finite semilattice, and let  $n:=|L|$ denote the number of its elements. If $L$ has at least $127\cdot 2^{n-8}$ subsemilattices, then it is a planar semilattice.  
\end{theorem}

Another variant of this result will be stated in Theorem~\ref{thmreform}.

\begin{remark}\label{remarksharp}
 For $n\geq9$, Theorem~\ref{thmmain} is sharp, since there exists an $n$-element non-planar semilattice with exactly $127\cdot 2^{n-8}-1$ subsemilattices. 
\end{remark}

\begin{remark}\label{remmsVnpL} Every semilattice with  \emph{at most seven elements} is planar, regardless the number of its subsemilattices. 
While $\slat{A_0}$ from Lemma~\ref{lemmacomputed} is an  8-element non-planar semilattice with $121$ subsemilattices, 
every \emph{eight-element} semilattice with at least $122=122\cdot 2^{8-8}$ subsemilattices is planar.
\end{remark}

This remark will be proved at the end of Section~\ref{sectfences}.

\begin{remark}\label{remarkmantissa}
Although the numbers $63.5\cdot 2^{n-7}$, 
$31.75\cdot 2^{n-6}$, $15.875\cdot 2^{n-5}$, \dots{} and  $254\cdot 2^{n-9}$, $508\cdot 2^{n-10}$, $1016\cdot 2^{n-11}$, \dots{} are all equal to $127\cdot 2^{n-8}$, we want to avoid fractions as well as large coefficients of powers of 2. This explains the formulation of Theorem~\ref{thmmain}.
\end{remark}

Our result is motivated by similar or analogous results for lattices and for congruences; see 
Ahmed and Horv\'ath~\cite{delbrineszter},
Cz\'edli~\cite{czgnotelatmanyC}, \cite{czglatmancplanar}, \cite{czgslatmanc}, and \cite{czg83},   Cz\'edli and Horv\'ath~\cite{czgkhe}, and Mure\c san and Kulin~\cite{kulinmuresan}.  In particular,  \cite{czg83} proves that if an $n$-element lattice $L$ has at least $83\cdot 2^{n-8}$ sublattices, then $L$ is planar.

\begin{remark} Clearly, an $n$-element semilattice $L$ can have at most $2^{n}-1$ subsemilattices, and it has so many subsemilattices if and only if $L$ is a chain. Chains are planar semilattices. Since, up to isomorphism, there are only finitely many $n$-element semilattices, we can let $\mu(n):=1+\max\{k:$ there is an $n$-element non-planar semilattice with exactly $k$ subsemilattices$\}$. Putting these three facts together, it follows trivially that every $n$-element semilattice with at least $\mu(n)$ subsemilattices is planar, and this result is sharp. So the novelty in this paper is that $\mu(n)$ is explicitly determined and it is given by a simple expression.
\end{remark}

\subsection*{Outline} Apart from half a page devoted to the proof of Remark~\ref{remmsVnpL} at the end of Section~\ref{sectfences}, the rest of the paper is devoted to the proof of Theorem~\ref{thmmain}. In particular, Section~\ref{sectionfirststeps} contains Theorem~\ref{thmreform}, which is a useful reformulation of Theorem~\ref{thmmain}, and Lemma~\ref{lemmakey}; both statements are worth separate mentioning here. In  Section~\ref{sectKR}, 
a deep theorem of Kelly and Rival~\cite{kellyrival} for \emph{lattices} is recalled and a related lemma, proved by our \emph{computer program}, is presented. While reading this section and the rest of the paper, Cz\'edli~\cite{czg83} should be near, since it contains some notation and figures that we need in the present paper.  The rest of the proof is given in Section~\ref{sectfences}.

\section{Another form of our result and some lemmas}\label{sectionfirststeps}
\subsection{Relative number of subuniverses}
A \emph{partial groupoid} is a structure $\slat A$ such that $A$ is a nonempty set and $\vee$ is a map from a subset $\dom {\vee}$ of $A^2$ to $A$. A \emph{subuniverse} of a partial groupoid  $\slat A$
is a subset $X$ of $A$ such that whenever $x,y\in X$ and $(x,y)\in\dom {\vee}$, then $x\vee y\in X$. The set of subuniverses of $\slat A$ will be denoted by $\Sub\slat A$. For a semilattice $\slat L$, $\Sub\slat L$ is usually called the \emph{lattice of subsemilattices} of $\slat L$. In spite of this terminology, note that the collection of subsemilattices is only  $\Sub\slat L\setminus\set{\emptyset}$.  We often abbreviate $\slat L$ and $\Sub\slat L$ as $L$ and $\Sub(L)$, respectively; sometimes, this convention is indicated by $L=\slat L$. Similar convention applies for posets, lattices, and partial groupoids. 
All semilattices, posets, lattices, and partial groupoids in this paper are automatically assumed to be \emph{finite} even if this is not repeated all the time. For more about these structures, the reader can resort to the monograph Gr\"atzer~\cite{ggglt}.

Since a large semilattice $L$ has many (more than $|L|$)  subsemilattices, it is reasonable to relate their number to the number $|L|$ of elements of $L$. Thus, motivated by Cz\'edli~\cite{czg83}, we will adhere to the following terminology and notation.

\begin{definition}
The \emph{relative number of subuniverses} of an $n$-element finite partial groupoid $\slat A$ is defined to be and denoted by
\[
\many\slat A:=|\Sub\slat A|\cdot 2^{8-n}.
\] 
Furthermore, we say that a finite semilattice $L$ has \emph{\smany{} subsemilattices} or, in other words, it has \emph{\smany{}  subuniverses} if $\many(L)>127$.
\end{definition}

Since $|\Sub(L)|$ is larger than the number of subsemilattices by 1, we can reformulate Theorem~\ref{thmmain} and  Remark~\ref{remarksharp} as follows.

\begin{theorem}\label{thmreform}
If $L$ is a finite semilattice such that $\many(L)>127$, then $L$ is planar. In other words, finite semilattices with \smany{} subsemilattices are planar. Furthermore, for every natural number $n\geq 9$, there exists an $n$-element semilattice $L$ such that $\many(L)=127$ and $L$ is not planar.
\end{theorem}

For partial groupoids $\silat{A_1}1$ and $\silat{A_2}2$, we say that 
$\silat{A_1}1$ is a \emph{weak partial subgroupoid} (or a  \emph{weak partial subalgebra}) of  $\silat{A_2}2$ if $A_1\subseteq A_2$, $\dom{\vee_1}\subseteq \dom{\vee_2}$, and $x\vee_1 y=x\vee_2 y$ holds for every $(x,y)\in\dom{\vee_1}$. 
The following easy lemma has been proved in Cz\'edli~\cite{czg83}; it is Lemma 2.3 there.

\begin{lemma}\label{lemmapartext}\
\begin{enumeratei}
\item\label{lemmapartexta}
If $B=\slat B$ is a weak partial subgroupoid of a finite partial groupoid $A=\slat A$, then $\many(B)\geq \many(A)$.  
\item\label{lemmapartextb} In particular, if $S=\slat S$ is a subsemilattice of a finite semilattice $L=\slat L$, then $\many(S)\geq \many(L)$.
\end{enumeratei}
\end{lemma}

We also need a deeper statement, which we formulate below.

\begin{lemma}[Key Lemma]\label{lemmakey}
Let $S=\slat S$ and $L=\slat L$ be finite semilattices, and assume that $S$ is a \emph{subposet} of $L$, i.e., $S\subseteq L$ and for any $x,y\in S$, we have $x\leqs y$ if and only if $x\leql y$. Then $\many(S) \geq \many(L)$.
\end{lemma}

Since subsemilattices are subposets, Lemma~\ref{lemmakey} implies part \eqref{lemmapartextb} of Lemma~\ref{lemmapartext}.
For a poset $P$, we will use the standard notation
\begin{equation*}
\text{$\Mi(P):=\{x\in P: x$ has exactly one upper cover$\}$.} 
\end{equation*}
The covering relation in $P$ will be denoted by $\prec_P$; so $x\prec_P y$ will mean that $|\set{z\in P: x\leq z\leq y}|=2$. For $X\subseteq P$, the \emph{set of upper bounds} of $X$ is denoted by $\ubin P(X):=\{y\in P: y\geq x$ for all $x\in X\}$. For $X=\set{x_1,\dots,x_n}$, we will write 
$\ubin P(x_1,\dots,x_n)$ rather than $\ubin P(\set{x_1,\dots,x_n})$.

\begin{proof}[Proof of Lemma~\ref{lemmakey}] For the sake of contradiction, suppose that the lemma fails. Then we can pick  semilattices $S$ and $L$ with minimal value of $|L\setminus S|$  such that $S$ is a subposet of $L$ and $\many(S)<\many(L)$. We know from Lemma~\ref{lemmapartext}\eqref{lemmapartextb} that $S$ is \emph{not} a subsemilattice of $L$. Hence, we can pick a \emph{minimal} element $j\in S$ such that $a\vee_S b\neq a\vee_L b$ for some $a,b\in S$. Let $d:=a\vee_L b\in L$. Clearly, $a$ is \emph{incomparable} with $b$ (in notation, $a\parallel b$), $d<j$, and $d\notin S$. Define $B:=S\cup\set d$. With the ordering inherited from $L$, $B=(B;\leq)$ is a subposet of $L$. If $d<_B x$ for some $x\in B$, then $x\in S\cap \ubin S(a,b)$ gives that $j=a\vee_S b\leq x$, whereby $d\prec_B j$ and, in addition, $j$ is the only upper cover of $d$ in $B$. This proves the first half of the following observation; the second half is an easy consequence of the first one.
\begin{equation}
\text{$d\prec_B j$,\,$d\in \Mi(B)$, and so, for $s\in S$, $s\parallel j\then s\parallel d$.}
\label{eqondnB}
\end{equation}  
We are going to show that $(B;\leq)$ is a (join-)semilattice. With the notation $\ofideal B d:=\set{x\in B: x\leq d}$,  we let 
$D:=S\cap\ofideal B d$. If $x,y\in B$ are comparable elements, then $x\vee_B y$ trivially exists and equals $x\vee_L y\in\set{x,y}$. We claim that whenever 
$x,y\in B$ and $x\parallel y$, then $x\vee_B y$ still exists and
\begin{align}
x\vee_B y&=x\vee_S j, &&\text{ if $y=d$ and $x\parallel d$},\label{alignjoina}\\
x\vee_B y&=j\vee_S y, &&\text{ if $x=d$ and $d\parallel y$},\label{alignjoinb}\\
x\vee_B y&=x\vee_S y, &&\text{ if $\set{x,y}\subseteq S$ and $x\vee_S y\neq j$},\label{alignjoinc}\\
x\vee_B y&=x\vee_S y=j, &&\text{ if $\set{x,y}\subseteq S$,  $\set{x,y}\not\subseteq D$ and $x\vee_S y = j$},\label{alignjoind}\\
x\vee_B y&=d, &&\text{ if $\set{x,y}\subseteq S$,  $\set{x,y}\subseteq D$ and $x\vee_S y = j$}.\label{alignjoine}
\end{align}
Note that the inclusion $\set{x,y}\subseteq S$ makes the assumption in \eqref{alignjoine} redundant; this inclusion occurs there only for emphasis. It suffices to show that for each of \eqref{alignjoina}--\eqref{alignjoine}, the element given right after ``$x\vee_B y=$'' is the smallest element of $\ubin B(x,y)$.

In order to verify \eqref{alignjoina}, assume that  $x\parallel d$. Then $d\notin \ubin B(x,d)$, and it follows from \eqref{eqondnB} that $\ubin B(x,d)=\ubin B(x,j)=\ubin S(x,j)$, whereby we conclude  \eqref{alignjoina}. Since the role of $x$ and $y$ is symmetric, we also conclude \eqref{alignjoinb}.

Next, assume that $\set{x,y}\subseteq S$ and $x\vee_S y\neq j$. If $x\vee_S y>j$ or $x\vee_S y\parallel j$, then $j\notin \ubin S(x,y)$ gives that $d\notin \ubin B(x,y)$, whereby $\ubin B(x,y)=\ubin S(x,y)$ and we obtain that $x\vee_B y$ exists and equals $x\vee_S y$. If $x\vee_S y<j$, then there are two cases. First, if $x\vee_S y<d$, then $x\vee_B y=x\vee_S y$ is clear. Second, assume that  $x\vee_S y\not<d$ while $x\vee_S y<j$. Since $x\vee_S y\in S$ but $d\notin S$, 
$x\vee_S y\neq d$. So $x\vee_S y\not\leq d$.
 The minimality of $j$ gives that $x\vee_S y= x\vee_L y$. Hence, $x\vee_L y= x\vee_S y\not\leq d$ yields that at least one of $x\leq d$ and $y\leq d$ fails,  whence $d\notin \ubin B(x,y)$. Consequently, $\ubin B(x,y)=\ubin S(x,y)$ and $x\vee_B y=x\vee_S y$ is clear again. This proves \eqref{alignjoinc}.

Since the assumptions in  \eqref{alignjoind} imply that $d\notin \ubin B(x,y)$, we have that $\ubin B(x,y)=\ubin S(x,y)$, whereby \eqref{alignjoind} follows.

Finally, $\set{x,y}\subseteq D$ and $x\vee_S y = j$, then $\ubin B(x,y)=\set d\cup\ubin S(x,y)= \set d\cup\offilter S j$ implies that $x\vee_B y=d$, proving  \eqref{alignjoine}. Now, as it was mentioned earlier, 
 \eqref{alignjoina}-- \eqref{alignjoine} imply that $B=(B;\leq)$ is a semilattice $\slat B$. Since $B$ is also a subposet of $L$ and $|L\setminus B|=|L\setminus S|-1$, the minimality of $|L\setminus S|$ yields that 
\begin{equation}
\many\slat B \geq \many\slat L.\label{eqBmrtkLts}
\end{equation}

Next, for a subset $Y$ of $S$, we denote by $[Y]_S$ the subsemilattice of $\slat S$ generated by $Y$. The notation $[Y]_B$ is understood analogously.
Consider the  map
\[\phi\colon\Sub\slat B\to\Sub\slat S,\text{ defined by } X\mapsto [X\setminus\set d]_S;
\]
our plan is to show that 
\begin{equation}
\text{each $Y\in\Sub\slat S$  has exactly one or two preimages with respect to $\phi$.}\label{eqtxtegyket}
\end{equation}
In order to do so, assume that $Y\in\Sub\slat S$. 
There are two cases to consider; assume first that  $j\notin Y$ and let $Z:=[Y]_B$. Let $Z_0:=Y$ and, for $i>0$, let 
\begin{equation}
\text{
$Z_{i+1}:=Z_i\cup\{x\vee_B y: x,y\in Z_i,\,\,x\parallel y\}$.  Then $Z=\bigcup_{i\in \nnul} Z_i$,}
\label{eqnXtzZz}
\end{equation} 
where $\nnul$ denotes the set of nonnegative integers. 
Since $Z_0\cap\set{j,d}=\emptyset$ and $j\notin Y$, 
only \eqref{alignjoinc} from the rules \eqref{alignjoina}--\eqref{alignjoine} can be applied when we compute $Z_1$ according to \eqref{eqnXtzZz}. However,  \eqref{alignjoinc}  does not produce any new element since $Z_0=Y\in\Sub\slat S$. So, $Z_1=Z_0$, and \eqref{eqnXtzZz} leads to $Y=Z_0=Z_1=Z_2=\dots = Z\in\Sub\slat B$. Since $\phi(Z)=\phi(Y)=[Y\setminus\set d]_S= [Y]_S=Y$, $Y$ has at least one preimage, $Z=Y$.

Now, let $X$ be an arbitrary preimage of $Y$. Since $Y=\phi(X)=[X\setminus \set d]_S$, we have that $X\setminus \set d\subseteq Y$, that is, $X\subseteq Y\cup\set{d}$. Thus, to show that $Y$ has at most two preimages, it suffices to show that $Y\subseteq X$, because then $X$ will necessarily belong to $\set{Y, Y\cup\set d}$. Suppose, for a contradiction, that $Y\not\subseteq X$, and pick an element $u\in Y\setminus X$. 
Then, using $X\subseteq Y\cup\set{d}$, we have that  
$u\in Y=\phi(X)=[X\setminus\set d]_S\subseteq [Y\setminus\set u]_S$. Hence, there is a smallest $k$ such that 
$u=y_1\vee_S\dots \vee_S y_k$ with some $y_1,\dots, y_k\in X\setminus \set d\subseteq Y\setminus\set u$. 
Since $Y\in\Sub\slat S$, 
\begin{equation}\left.
\parbox {8cm}{all the joins 
$y_1\vee_S y_2$, $y_1\vee_S y_2\vee_S y_3 = (y_1\vee_S y_2)\vee_S y_3$,  
$y_1\vee_S y_2\vee_S y_3\vee_S y_4 = (y_1\vee_S y_2\vee_S y_3)\vee_S y_4$, \dots, $y_1\vee_S\dots \vee_S y_k=u$ belong to $Y$, 
}\right\}\label{pbxallTjNsS}
\end{equation}
and none of them is $j$ since $j\notin Y$. By the minimality of $k$, all the outer joins above apply to incomparable joinands.
Thus, only \eqref{alignjoinc} of the five computational rules applies to these outer joins, whereby
\begin{equation}\left.
\parbox {8cm}{
the joins in \eqref{pbxallTjNsS} equal the joins 
$y_1\vee_B y_2$, $y_1\vee_B y_2\vee_B y_3 = (y_1\vee_B y_2)\vee_B y_3$,  
$y_1\vee_B y_2\vee_B y_3\vee_B y_4 = (y_1\vee_B y_2\vee_B y_3)\vee_B y_4$, \dots, $y_1\vee_B\dots \vee_B y_k=u$.}\right\} \label{pbxallTjNinT}
\end{equation}
Since all the $y_1,\dots, y_k$ belong to $X$, so does $u$, which is a contradiction. Thus,
 \eqref{eqtxtegyket} holds for the particular case $j\notin Y$.

Second, we assume that $j\in Y$. Let $Z:=Y\cup\set d$. Rules  \eqref{alignjoina}--\eqref{alignjoine}
yield that for incomparable $x,y\in Z$,
the join $x\vee_B y$ belongs to $\set{d, x\vee_S j, j\vee_S y, x\vee_S y} \subseteq \set d\cup Y=Z$, since $Y$ is $\vee_S$-closed, $j\in Y$, and $d\in Z$. Hence, $Z\in\Sub\slat B$. Since $\phi(Z)=[Z\setminus\set d]_S=[Y]_S=Y$, we have obtained that $Y$ has at least one preimage, $Z$. Now, let $X\in\Sub\slat B$ be an arbitrary preimage of $Y$. We claim that 
\begin{equation}
Y\setminus\set j \subseteq X\subseteq Y\cup \set d.
\label{eqngykzrfGtm}
\end{equation}
The second inclusion is clear by the definition of $\phi$. Suppose, for a contradiction, that the first inclusion fails, and pick an element $u\in Y\setminus\set j$ such that $u\notin X$. Note that the earlier meaning of $u$, given before \eqref{pbxallTjNsS}, is no longer valid. Observe that 
$u\in Y=\phi(X)=[X\setminus\set d]_S$ and $X\subseteq Y\cup \set d$ imply that there is a smallest $k$ such that $u$ is of the form $u=y_1\vee_S\dots \vee_S y_k$ with some $y_1,\dots, y_k\in X\cap(Y\setminus\set{u})$. 
The equalities listed in  
\eqref{pbxallTjNsS} and \eqref{pbxallTjNinT} are understood for the present situation. If none of the joins in \eqref{pbxallTjNsS} equals $j$, then 
\eqref{pbxallTjNinT} is still valid and leads to $u\in X$, which is a contradiction. 
If one of the joins in \eqref{pbxallTjNsS} is $j$, then this join is not the last one (which gives $u$), and 
this join can change to $d$ in \eqref{pbxallTjNinT} 
by rule \eqref{alignjoine}. 
More precisely, it may happen that 
$(y_1\vee_S\dots \vee_S y_{i-1})\vee_S y_i$ equals $j$ in \eqref{pbxallTjNsS} for some $i<k$ but 
$(y_1\vee_B\dots \vee_
B y_{i-1})\vee_B y_i=d$ in \eqref{pbxallTjNinT}; note that $i$ is uniquely determined since $k$ was the least possible number for \eqref{pbxallTjNsS} and so the joins in \eqref{pbxallTjNsS} form a strictly increasing sequence.  
By the minimality of $k$, we have that $j\parallel y_{i+1}$. Hence, $d\parallel y_{i+1}$ by \eqref{eqondnB}. It follows from \eqref{alignjoinb} that 
\[(y_1\vee_B\dots\vee_B y_i)\vee_B y_{i+1}=d\vee_B y_{i+1}=j\vee_S y_{i+1}=(y_1\vee_S\dots\vee_S y_i)\vee_S y_{i+1}.
\]
We have seen that there can be at most one $i<k$ such that the $i$-th join in \eqref{pbxallTjNsS} and that in \eqref{pbxallTjNinT} are different, but we still have that $u=y_1\vee_B\dots \vee_B y_n\in X$, which is a contradiction. This proves \eqref{eqngykzrfGtm}.

If $Y\setminus \set j\in \Sub\slat S$, then $j$ must belong to $X$, because otherwise \eqref{eqngykzrfGtm} would lead to $j\in Y=\phi(X)=[X\setminus\set d]_S\subseteq [Y\setminus\set j]_S=Y\setminus\set j$, a contradiction. Hence, if $Y\setminus \set j\in \Sub\slat S$, then $j\in X$ combined with  \eqref{eqngykzrfGtm} allow at most two possible sets $X$. 
Hence, we can assume that in addition to $j\in Y$, we have that $Y\setminus\set j\notin\Sub\slat S$.  Then, since $Y$ is $\vee_S$-closed but $Y\setminus\set j$ is not, there are $x,y\in Y\setminus \set j$ such that $j=x\vee_S y$. Clearly, $x\parallel y$.  We know from \eqref{eqngykzrfGtm} that $x$ and $y$ belong to $X$, whence $x\vee_B y\in X$. 
It follows from  \eqref{alignjoind} and \eqref{alignjoine} that $x\vee_B y\in\set{j,d}$.
If  $x\vee_B y=j$, then $j\in X$ and $X\in \set{Y, Y\cup \set d}$ by \eqref{eqngykzrfGtm}, so there are at most two preimages $X$ of $Y$. Similarly, if $x\vee_B y=d$, then $d\in X$ and $X\in \set{(Y\cup\set d)\setminus\set j, Y\cup \set d}$ by \eqref{eqngykzrfGtm} and, again there are at most two preimages $X$ of $Y$. Hence, we have shown that if $j\in Y$, then $Y$ has one or two preimages.

Now, after that all cases have been considered, \eqref{eqtxtegyket} has been proved. As a particular case of \eqref{eqtxtegyket}, we know that $\phi$ is a surjective map. It is a trivial consequence of \eqref{eqtxtegyket} that 
$2\cdot |\Sub\slat S| \geq |\Sub\slat B|$. Dividing this inequality by $2\cdot 2^{|S|-8} = 2^{|B|-8}$, we obtain that $\many \slat S\geq \many\slat B$.
This inequality and \eqref{eqBmrtkLts} yield that $\many\slat S\geq \many \slat L$, contradicting our initial assumption and completing the proof of Lemma~\ref{lemmakey}.
\end{proof}

\section{A deep result of D.\ Kelly and I.\ Rival and a computer program}\label{sectKR}
For a poset $P$, its dual will be denoted by $\dual P$.
With reference to Kelly and Rival~\cite{kellyrival} or, more conveniently, to Figures~1--5 of Cz\'edli~\cite{czg83}, where Kelly and Rival's theorem is recalled, the \emph{Kelly--Rival list} of \emph{lattices} is defined as follows.
\begin{equation}
\krlist:=\set{A_n, E_n, \dual{E_n}, F_n, G_n, H_n: n\geq 0} \cup \set{B,\dual B, C, \dual C, D,\dual D}.
\label{eqKRlist}
\end{equation}
Note that $A_n$, $F_n$, $G_n$, and $H_n$ are selfdual lattices. Our main tool is the following theorem. 

\begin{theorem}[Kelly--Rival Theorem, taken from Kelly and Rival~\cite{kellyrival}]\label{thmKR}
A finite lattice $L$ is planar if and only if for every $X\in\krlist$,  $X$ is not a \emph{subposet} of $L$. 
\end{theorem}

Note that the original version of this theorem in \cite{kellyrival} says more by stating that $\krlist$ is the unique minimal list that makes the theorem work but we will not use this fact. 

When working on paper \cite{czg83}, 
 the author has developed a straightforward \emph{computer program} under Windows 10.  This program, called \emph{subsize}, is downloadable from the author's website. 
Using the straightforward trivial algorithm, 
the program computes $|\Sub(A; F)|$ for an arbitrary partial algebra $(A;F)$, provided $|A|$ is small and it has only (at most) binary operations. The following lemma as well as  some other statements in the rest of the paper are proved by this program; an appropriate input file (a single file for all these statements) is  available from the author's website; see the list of publications there. While the program is reliable up to the author's best knowledge and one can easily write another computer program, the reader may want to (but need not) check the input file for correctness. (Note that the notation used in the input file is taken from the figures in Cz\'edli~\cite{czg83}.) The output file, from which the input file can easily be recovered, is an appendix of the arXiv version\footnote{\emph{This} is the arXiv version.} of the present paper.

The join-semilattices occurring in the lemma below are the semilattice reducts of the ``small''  lattices occurring in the Kelly--Rival list $\krlist$. Note that, for any lattice $X$, $\slat {\dual X}$ is the same as  $(X;\wedge)$; this trivial fact made it easier to produce most parts of the input file from one of the input files that go with \cite{czg83}.

\begin{lemma} \label{lemmacomputed}
 $\many(A_0)=\many\slat{A_0}=122$, 
 $\many\slat{B}=108$, 
 $\many\slat{\dual B}=114$, 
 $\many\slat{C}=123$, 
 $\many\slat{\dual C}=113$,
 $\many\slat{D}=116$, 
 $\many\slat{\dual D}=124$,
 $\many\slat{E_0}=114$,
 $\many\slat{\dual E_0}=110$,
 $\many\slat{E_1}=79.75$,
 $\many\slat{\dual E_1}=84.5$,
 $\many\slat{F_0}=127$, 
 $\many\slat{F_1}=88.75$,
 $\many\slat{G_0}=98.75$, and 
 $\many\slat{H_0}=99.5$.
\end{lemma}

\section{Fences, a snake, and the end of the proof}\label{sectfences}
We need the following four posets. The \emph{enriched 8-element fence} and the \emph{10-element snake} are given in Figure~\ref{figa}. (In spite of its name, the enriched 8-element fence consists of ten elements.) If we remove  all the black-filled elements, labeled by $i$, from Figure~\ref{figa}, then we obtain the \emph{9-element up-fence}, its dual, the \emph{9-element down-fence}, and the \emph{8-crown}.

\begin{lemma}\label{lemmasnakefence}
If $L$ is a finite join-semilattice such that $\many(L)>127$ (in other words, if $L$ has \smany{} subsemilattices), then none of the enriched 8-element fence, the 8-crown, the 9-element up-fence, the 9-element down-fence, and the 10-element snake is a \emph{subposet} of $L$.
\end{lemma}

\begin{proof} Let $\slat L$ be a finite join-semilattice such that $\many(L)>127$. Unless otherwise stated, the join $\vee$ is understood in $\slat L$. The notation for the elements given in Figure~\ref{figa} will be in effect.

\begin{figure}[htb] 
\centerline
{\includegraphics[scale=0.9]{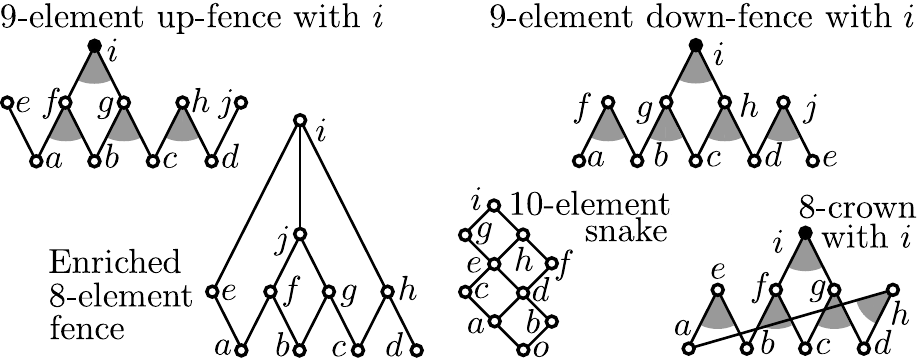}}
\caption{Fences, the 8-crown, and the 10-element snake; the empty-filled elements only
\label{figa}}
\end{figure}%

For the sake of contradiction, suppose that the 9-element down-fence, denoted here by $X$, is a subposet of $L$. Clearly, $a\vee b\leq f$. If 
$a\vee b < f$, then we can replace $f$ by $f':=a\vee b$; all the previous comparabilities and incomparabilities that are true for $f$ will remain true for $f'$. For example, if we had $f'\leq g$ then $a\leq f'$ would lead to $a\leq g$ by transitivity, but $a\not\leq g$ in $X$. In the next step, we omit $f$ and rename $f'$ as $f$. Hence, we can assume that $a\vee b=f$ in $L$. We can continue similarly, and finally we can assume that 
\begin{equation}
a\vee b=f,\quad b\vee c=g,\quad c\vee d=h,\quad d\vee e=j,\quad g\vee h=i.
\label{eqdnFnsgY}
\end{equation}
Note that these assumptions are indicated by grey-filled angles in Figure~\ref{figa}. Note also that these assumptions should be made in the order they are listed above; for example, we have to ``fix'' the joins at $g$ and $h$ before introducing the additional element $i:=g\vee h$. Clearly, 
\eqref{eqdnFnsgY} implies that 
\begin{equation}
b\vee h=i \text{ and } g\vee d=i;
\label{eqdnzFnsgpt}
\end{equation}
for example, $ b\vee h=b\vee c\vee h=g\vee h=i$.
Armed with the seven equalities given in \eqref{eqdnFnsgY} and \eqref{eqdnzFnsgpt}, $X\cup\set i$ turns into a partial groupoid $(X\cup\set i;\vee_X)$, which is a weak partial subalgebra of $\slat L$. Using our computer program, we obtain that $\many(X\cup\set i;\vee_X)=123.5$. So, by Lemma~\ref{lemmapartext}, $123.5\geq \many\slat L$, contradicting $\many\slat L>127$. Therefore, the 9-element down-fence cannot be a subposet of $\slat L$.

The argument for the 8-crown is almost the same; the only difference is that \eqref{eqdnFnsgY} and \eqref{eqdnzFnsgpt} are replaced by 
\begin{equation}
a\vee b=e,\quad a\vee d=h,\quad b\vee c=f,\quad c\vee d=g,\quad f\vee g=i.
\label{eqdnqTzsgY}
\end{equation}
and
\begin{equation}
b\vee g=i \text{ and } f\vee d=i, 
\label{eqdnzFnspknpW}
\end{equation}
and, in addition, now the $\many$-value of the partial groupoid $\set{a,b,\dots,i}$ is 125.

The treatment for the 9-element up-fence, denoted here by $Y$, is a bit more complex, but the argument begins in the same way as above. After modifying $f$, $g$, and $h$,  if necessary, and letting $i:=f\vee g$, we obtain a weak partial subgroupoid $(Y\cup\set i;\vee_Y)$ of $\slat L$, where
$\vee_Y$ and its domain are described by
\begin{align}
a\vee_Y b=f,\quad b\vee_Y c=g,\quad  c\vee_Y d=h,\quad  f\vee_Y g=i,\label{alYpra}\\
 a\vee_Y g=i,\quad\text{and}\quad f\vee_Y c=i.\label{alYprb}
\end{align}
Note that, in $L$, $a\vee g=a\vee b\vee g =f\vee g=i$, which explains the first equality in \eqref{alYprb}; the second one is explained similarly. In Figure~\ref{figa}, \eqref{alYpra} is visualized by grey-filled angles. 
Unfortunately, $\many(Y\cup\set i;\vee_Y)=137$ is rather large to draw any conclusion.  Hence, we need to deal with two cases. First, assume that $h\vee j=i$ in $L$. Then we add $h\vee_Y j =i$ and its consequence, $c\vee_Y j=i$  (explained by 
$c\vee j=c\vee d\vee j=h\vee j$) to the domain of $\vee_Y$, and $\many(Y\cup\set i;\vee_Y)$ turns out to be 122. Second, assume that $h\vee j=:k$ is distinct from $i$. Then we add $h\vee_Y j =k$ and its consequence, $c\vee_y j=k$  (explained again by 
$c\vee j=c\vee d\vee j=h\vee j$) to the domain of $\vee_Y$, and we obtain that $\many(Y\cup\set{i,k};\vee_Y)=114.25$. In both cases, we have obtained a weak partial subgroupoid of $\slat L$ such that the $\many$-value of this subgroupoid is at most 122. Hence, by Lemma~\ref{lemmapartext}, $\many\slat L$ is at most 122, contradicting our assumption that $\many\slat L>127$. Therefore, the 9-element up-fence cannot be a subposet of $\slat L$.

Next, let $Z$ denote the enriched 8-element fence. For the sake of contradiction, suppose that $Z$ is a subposet of $\slat L$. Observe that $Z$ is a join-semilattice, whereby the computer program, Lemma~\ref{lemmakey}, and our assumption on $\slat L$ yield that $78=\many\slat Z\geq \many\slat L>127$, which is a contradiction. Thus, the enriched 8-element fence cannot be a subposet of $\slat L$. Neither can the 10-element snake, because its $\many$-value is $125.5$  by the program and the same reasoning applies.  
\end{proof}

Now, we are in the position to complete the proof of our result.

\begin{proof}[Proof of Theorem~\ref{thmreform}]
Let $\slat L$ be a finite semilattice with $\many\slat L>127$. For the sake of contradiction, suppose that $L$ is not planar. Regardless whether $L$ has a smallest element or not, add a new bottom element 0 to $L$ and denote by $\poset{\plunul L}$ the poset we obtain in this way. So $\poset{\plunul L}$ is the disjoint union of $\set{0}$ and $L$, and $0<x$ for all $x\in L$. In another terminology,  
\begin{equation}
\text{$\poset{\plunul L}$ is the ordinal sum of the singleton poset and $L$.}
\label{eqtxtLNll}
\end{equation}
Clearly, 
with the ordering just defined,  $\slat{\plunul L}$ is also a join-semilattice; in fact, $\poset{\plunul L}$ is a lattice. Since $\Sub\slat{\plunul L}$ is the disjoint union of $\Sub\slat L$ and $\set{X\cup\set 0: X\in\Sub\slat L}$, it follows that 
\begin{equation}
\many\slat{\plunul L}=\many\slat L>127.
\label{eqplnlbBg}
\end{equation}
If $\slat{\plunul L}$ had a  diagram with non-crossing edges, then after deleting all edges starting from 0, we would obtain a planar digram of $\slat L$, and this would contradict our assumption on $\slat L$. Hence, $\lattice{\plunul L}$, that is, $\poset{\plunul L}$ is a non-planar lattice.
By Theorem~\ref{thmKR}, the Kelly--Rival Theorem, there exists a lattice $X=\poset X\in \krlist$, see \eqref{eqKRlist}, such that $\slat X$ is a subposet of $\slat{\plunul L}$. 

If $X=F_0$, then (the Key) Lemma~\ref{lemmakey} together with Lemma~\ref{lemmacomputed} give that
$\many\slat {\plunul L}\leq 127$, contradicting \eqref{eqplnlbBg}. If $X$ is another lattice occurring in Lemma~\ref{lemmacomputed}, then the contradiction is even bigger since $\many\slat X$ is smaller than 127.

If $X=A_1$ or $X\in\set{A_2, A_3, A_4,\dots}$, then 
$X$ and, thus, $\slat{\plunul L}$ contains the 8-crown or the 9-element down-fence as a subposet, which contradicts Lemma~\ref{lemmasnakefence}. If $X=E_n$ or $X=\dual E_n$ for some $n\geq 2$, then 
the 9-element up-fence or the 9-element down-fence
is a subposet of $\slat{\plunul L}$, and Lemma~\ref{lemmasnakefence} gives a contradiction again.
Since the enriched 8-element fence is a subposet of $F_2$ and each of $F_3$, $F_4$, \dots contains a 9-element up-fence as a subposet, Lemma~\ref{lemmasnakefence} excludes that $X\in\set{F_n: n\geq 2}$. Since the 10-element snake is a subposet of each of the $G_n$,  $n\geq 0$, we obtain from  Lemma~\ref{lemmasnakefence} that $X$ is not of the form $G_n$. (Note that $G_0$ is also taken care of by Lemma~\ref{lemmacomputed}.)  Finally, the possibility $X=H_n$ for some $n\geq 1$ is excluded again by the same lemma, since each of these $H_n$ contains the 10-element snake as a subposet. 

All $X\in \krlist$ have been ruled out, but this contradicts the existence of such an $X$. This proves the first sentence of Theorem~\ref{thmreform}.

Finally, we prove by induction that for $n\geq 9$, there exists an $n$-element \emph{non-planar} semilattice $\slat{L_n}$ such that $\many\slat{L_n}=127$. Define $\slat{L_9}:=\slat{F_0}$; we know from Lemma~\ref{lemmacomputed} that  $\many\slat{L_9}=127$. Since $\poset{L_9}=\poset{F_0}$ is a lattice, it is non-planar be the (Kelly--Rival) Theorem~\ref{thmKR}. 
For $n>9$, we let 
$\slat{L_{n}}:=\slat{\plunul L_{n-1}}$. Using the equality from \eqref{eqplnlbBg}, we obtain that 
 $\many\slat{L_n}=\many\slat{\plunul L_{n-1}}=
\many\slat{ L_{n-1}}=127$. Since $\slat{L_{n-1}}$
is non-planar, so is  $\slat{L_{n}}$; see the obvious sentence right after \eqref{eqplnlbBg}. Thus, we have constructed a semilattice $\slat{L_n}$ with the required properties for all $n\geq 9$ by induction. 
This completes the proof of Theorem~\ref{thmreform}.
\end{proof}

\begin{proof}[Proof of Theorem~\ref{thmmain}] Apply the first two sentences of Theorem~\ref{thmreform}
\end{proof}

\begin{proof}[Proof of Remark~\ref{remarksharp}]
Apply the last sentence of Theorem~\ref{thmreform}.
\end{proof}

\begin{figure}[htb] 
\centerline
{\includegraphics[scale=0.9]{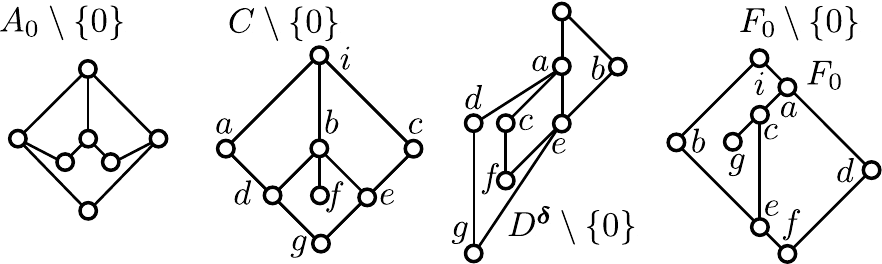}}
\caption{These join-semilattices are planar
\label{figb}}
\end{figure}%

\begin{proof}[Proof of Remark~\ref{remmsVnpL}]For the sake of contradiction, suppose that $L$ is a non-planar join-semilattice with at most seven elements. Then 
$\plunul L$, see \eqref{eqtxtLNll}, is a non-planar lattice and $|\plunul L|\leq 8$. Since $A_0$ is the only member of the Kelly--Rival list $\krlist$  with at most eight elements, it follows from (the Kelly--Rival) Theorem~\ref{thmKR} that $\plunul L=A_0$. But this is a contradiction since then $L=A_0\setminus\set 0$ is a planar poset; see Figure~\ref{figb}.

Since $|A_0|=8$, the equality  $\many\slat{A_0}=122$ from Lemma~\ref{lemmacomputed} means that $\slat{A_0}$ has exactly $121$ subsemilattices. Since $\poset{A_0}$ is also a lattice,  $\slat{A_0}$ is non-planar by (the Kelly--Rival) Theorem~\ref{thmKR}.

For the sake of contradiction again, suppose that $\slat L$ is an eight-element \emph{non-planar} semilattice with at least 122 subsemilattices, that is, $|\Sub\slat{L}|\geq 123$.  Then, as in \eqref{eqplnlbBg},  
$\many\slat {\plunul L}= \many\slat {L}\geq 123$ and $\poset{\plunul L}$ is a  non-planar lattice.  By 
(the Kelly--Rival) Theorem~\ref{thmKR}, some $X\in \krlist$ is a subposet of $\plunul L$. Clearly, $|X|\leq |\plunul L|=9$, and it follows from Lemma~\ref{lemmakey} that $\many\slat X\geq \many\slat {\plunul L}\geq 123$. 
Comparing these inequalities with Lemma~\ref{lemmacomputed}, which takes care of all at most nine-element members of $\krlist$ (and some larger members of $\krlist$), it follows that $X$ belongs to $\set{C, \dual D, F_0}$. 
Since $|X|=9=|\plunul L|$ and $X\subseteq \plunul L$ imply $\plunul L=X$,  we conclude that $\slat{L}$ is  obtained from some $\slat X\in \set{\slat C, \slat{\dual D}, \slat{F_0}}$ by removing its bottom element. Hence, $\slat L$ is one of the eight-element \emph{planar} join-semilattices given on the right of Figure~\ref{figb}, which is a contradiction.
\end{proof}

\clearpage
\centerline{\bf{{\LARGE APPENDIX}}}
\normalfont
\bigskip
\def\appendixhead{{Cz\'edli: One hundred twenty-seven subsemilattices / Appendix}}
\markboth\appendixhead\appendixhead

\noindent The rest of the paper is an appendix, which consists of the output file mentioned in the paragraph following Theorem~\ref{thmKR}.

\normalfont
\begin{verbatim}
Version of the input file: June 27, 2019
SUBSIZE version June 30, 2019 (started at 1:28:52) reports:
 [ Supported by the Hungarian Research Grant KH 126581,
                                (C) Gabor Czedli, 2018 ]


Semilattices and partial groupoids for the paper
One hundred twenty-seven subsemilattices and planarity

A_0 is the 8-element boolean join-semilattice with edges
 oa ab oc Ai Bi Ci aB aC bA bC cA cB
|A|=8, A(without commas)={oiabcABC}. Constraints:
a+b=C a+c=B a+A=i  b+c=A b+B=i  c+C=i  A+B=i   A+C=i   B+C=i
Result for A=A_0:  |Sub(A)| = 122, whence
sigma(A) = |Sub(A)|*2^(8-|A|) =  122.0000000000000000 .

Eight-crown
 ae ah be bf cf cg dg dh and also fi gi
|A|=9, A(without commas)={abcdefghi}. Constraints:
a+b=e a+d=h b+c=f c+d=g f+g=i
 b+g=i ; since b+g=b+c+g=f+g=i
 f+d=i ; since f+d=f+c+d=f+g=i
Result for A=Eight-crown:  |Sub(A)| = 250, whence
sigma(A) = |Sub(A)|*2^(8-|A|) =  125.0000000000000000 .

B is the join-semilattice with edges 
 oa ab oc od ae be bf bg cf dg ei fi gi
|A|=9, A(without commas)={oiabcdefg}. Constraints:
a+b=e  a+c=i  a+d=i  a+f=i  a+g=i  b+c=f  b+d=g  c+d=i
c+e=i c+g=i d+e=i  d+f=i e+f=i  e+g=i f+g=i
Result for A=B:  |Sub(A)| = 216, whence
sigma(A) = |Sub(A)|*2^(8-|A|) =  108.0000000000000000 .

The DUAL of B above; the edges of B are
 oa ab oc od ae be bf bg cf dg ei fi gi
|A|=9, A(without commas)={oiabcdefg}. Constraints:
a@b=o a@c=o  a@d=o  a@f=o  a@g=o  b@c=o  b@d=o  c@d=o c@e=o
c@g=o  d@e=o  d@f=o   e@f=b  e@g=b   f@g=b
Result for A=dual-B:  |Sub(A)| = 228, whence
sigma(A) = |Sub(A)|*2^(8-|A|) =  114.0000000000000000 .

C is the joinsemilattice with edges 
 ai bi ci da db eb ec fb gd ge og of
|A|=9, A(without commas)={oiabcdefg}. Constraints:
a+b=i a+c=i a+e=i a+f=i   b+c=i  c+d=i  c+f=i  d+e=b  d+f=b
e+f=b   f+g=b
Result for A=C:  |Sub(A)| = 246, whence
sigma(A) = |Sub(A)|*2^(8-|A|) =  123.0000000000000000 .

The DUAL of C above; the edges of C are
 ai bi ci da db eb ec fb gd ge og of
|A|=9, A(without commas)={oiabcdefg}. Constraints:
a@b=d a@c=g a@e=g a@f=o  b@c=e  c@d=g c@f=o  d@e=g d@f=o
e@f=o   f@g=o
Result for A=dual-C:  |Sub(A)| = 226, whence
sigma(A) = |Sub(A)|*2^(8-|A|) =  113.0000000000000000 .

D is the join-semilattice with edges 
 oa ob ac ae ad be cf dg ef eg fi gi
|A|=9, A(without commas)={oiabcdefg}. Constraints:
a+b=e   b+c=f b+d=g c+d=i c+e=f c+g=i  d+e=g d+f=i  f+g=i
Result for A=D:  |Sub(A)| = 232, whence
sigma(A) = |Sub(A)|*2^(8-|A|) =  116.0000000000000000 .

The DUAL of D above; the edges of D are
 oa ob ac ae ad be cf dg ef eg fi gi
|A|=9, A(without commas)={oiabcdefg}. Constraints:
a@b=o   b@c=o  b@d=o  c@d=a c@e=a c@g=a  d@e=a d@f=a  f@g=e
Result for A=dual-D:  |Sub(A)| = 248, whence
sigma(A) = |Sub(A)|*2^(8-|A|) =  124.0000000000000000 .

E_0 is the join-semilattice with edges 
 ai bi ci db ea ed fd fc gb oe of og
|A|=9, A(without commas)={oiabcdefg}. Constraints:
a+b=i a+c=i a+d=i a+f=i a+g=i  b+c=i  c+d=i c+e=i c+g=i
d+g=b e+f=d e+g=b  f+g=b
Result for A=E_0:  |Sub(A)| = 228, whence
sigma(A) = |Sub(A)|*2^(8-|A|) =  114.0000000000000000 .

The DUAL of E_0 above; the edges of E_0 are  
 ai bi ci db ea ed fd fc gb oe of og
|A|=9, A(without commas)={oiabcdefg}. Constraints:
a@b=e a@c=o a@d=e a@f=o a@g=o  b@c=f c@d=f c@e=o c@g=o
d@g=o e@f=o e@g=o  f@g=o
Result for A=dual-E_0:  |Sub(A)| = 220, whence
sigma(A) = |Sub(A)|*2^(8-|A|) =  110.0000000000000000 .

E_1 is the join-semilattice with edges
 ai bi ca da ei fb fc gc gd hd he ja of og oh oj 
|A|=11, A(without commas)={oiabcdefghj}. Constraints:
a+b=i a+e=i  b+c=i b+d=i b+e=i b+g=i b+h=i b+j=i
c+d=a c+e=i c+h=a c+j=a  d+e=i d+f=a d+j=a
e+f=i e+g=i e+j=i  f+g=c f+h=a f+j=a
g+h=d g+j=a  h+j=a
Result for A=E_1:  |Sub(A)| = 638, whence
sigma(A) = |Sub(A)|*2^(8-|A|) =   79.7500000000000000 .

The DUAL of E_1 above; the edges of E_1 are
 ai bi ca da ei fb fc gc gd hd he ja of og oh oj 
|A|=11, A(without commas)={oiabcdefghj}. Constraints:
a@b=f a@e=h  b@c=f b@d=o b@e=o b@g=o b@h=o b@j=o
c@d=g c@e=o c@h=o c@j=o
d@e=h d@f=o d@j=o  e@f=o e@g=o e@j=o
f@g=o f@h=o f@j=o  g@h=o g@j=o  h@j=o
Result for A=dual-E_1:  |Sub(A)| = 676, whence
sigma(A) = |Sub(A)|*2^(8-|A|) =   84.5000000000000000 .

F_0 is the join-semilattice with edges
 ai bi ca da eb ec fe fd gc of og
|A|=9, A(without commas)={oiabcdefg}. Constraints:
a+b=i  b+c=i b+d=i b+g=i  c+d=a  d+e=a  d+g=a  e+g=c  f+g=c
Result for A=F_0:  |Sub(A)| = 254, whence
sigma(A) = |Sub(A)|*2^(8-|A|) =  127.0000000000000000 .

F_1 is the join-semilattice with edges 
 oa od ab ac ah be bf cf cg dg ei fj gj hj ji
|A|=11, A(without commas)={oiabcdefghj}. Constraints:
a+d=g   b+c=f b+d=j b+g=j b+h=j  c+d=g c+e=i c+h=j
d+e=i d+f=j d+h=j  e+f=i e+g=i e+h=i e+j=i
f+g=j f+h=j  g+h=j
Result for A=F_1:  |Sub(A)| = 710, whence
sigma(A) = |Sub(A)|*2^(8-|A|) =   88.7500000000000000 .

G_0 is the join-semilattice with edges 
 oa ob ac ad ag bd ce de df eh ej fj gj hi ji
|A|=11, A(without commas)={oiabcdefghj}. Constraints:
a+b=d  b+c=e b+g=j  c+d=e c+f=j c+g=j  d+g=j  e+f=j e+g=j
f+g=j f+h=i  g+h=i  h+j=i
Result for A=G_0:  |Sub(A)| = 790, whence
sigma(A) = |Sub(A)|*2^(8-|A|) =   98.7500000000000000 .

H_0 is the join-semilattice with edges
 oa ob oc ad bd be bh cg df dg eg fi gi hi
|A|=10, A(without commas)={oiabcdefgh}. Constraints:
a+b=d a+c=g a+e=g a+h=i  b+c=g  c+d=g c+e=g c+f=i c+h=i
d+e=g d+h=i  e+f=i e+h=i  f+g=i f+h=i   g+h=i
Result for A=H_0:  |Sub(A)| = 398, whence
sigma(A) = |Sub(A)|*2^(8-|A|) =   99.5000000000000000 .

NO SUCCESS: 8fence is the join-semilattice with edges
 ae be bf cf cg dg dh fi gi
|A|=9, A(without commas)={abcdefghi}. Constraints:
a+b=e b+c=f c+d=g f+g=i
 b+g=i ; since b+g=b+c+g=f+g=i
 f+d=i ; since f+d=f+c+d=f+g
Result for A=8fence & i:  |Sub(A)| = 274, whence
sigma(A) = |Sub(A)|*2^(8-|A|) =  137.0000000000000000 .

9-element down-fence (and i) (a partial groupoid)
 af bf bg cg ch dh dj ej gi hi
|A|=10, A(without commas)={abcdefghji}. Constraints:
a+b=f b+c=g c+d=h d+e=j g+h=i
 b+h=i ; since b+h=b+c+h=g+h=i
 g+d=i ; since g+d=g+c+d=g+h=i
Result for A=9-element down-fence (and i):  |Sub(A)| = 494, whence
sigma(A) = |Sub(A)|*2^(8-|A|) =  123.5000000000000000 .

9-element up-fence; first attempt that fails
 ae af bf bg cg ch dh dj
|A|=10, A(without commas)={abcdefghji}. Constraints:
a+b=f b+c=g c+d=h f+g=i
 a+g=i ; since a+g=a+b+g=f+g=i
 f+c=i ; since f+c=f+b+c=f+g=i
Result for A=9-element up-fence &i(failure):  |Sub(A)| = 548, whence
sigma(A) = |Sub(A)|*2^(8-|A|) =  137.0000000000000000 .

9-element up-fence (a partial groupoid, Case 1)
 ae af bf bg cg ch dh dj  fi gi, and now hi ji for Case 1
|A|=10, A(without commas)={abcdefghji}. Constraints:
a+b=f b+c=g c+d=h f+g=i
 a+g=i ; since a+g=a+b+g=f+g=i
 f+c=i ; since f+c=f+b+c=f+g=i
     h+j=i ; Case 1
      c+j=i ; since c+j=c+d+j=h+j=i
Result for A=9-element up-fence & i; Case 1:  |Sub(A)| = 488, whence
sigma(A) = |Sub(A)|*2^(8-|A|) =  122.0000000000000000 .

9-element up-fence (a partial groupoid, Case 2)
 ae af bf bg cg ch dh dj, and now hk jk for Case2
|A|=11, A(without commas)={abcdefghjik}. Constraints:
a+b=f b+c=g c+d=h f+g=i
 a+g=i ; since a+g=a+b+g=f+g=i
 f+c=i ; since f+c=f+b+c=f+g=i
     h+j=k ; Case 2
      c+j=k ; since c+j=c+d+j=h+j=k
Result for A=9-element up-fence&i k; Case 2:  |Sub(A)| = 914, whence
sigma(A) = |Sub(A)|*2^(8-|A|) =  114.2500000000000000 .

Enriched 8-element fence (a join-semilattice)
 ae af bf bg cg ch dh fj gj ei ji hi
|A|=10, A(without commas)={abcdefghji}. Constraints:
a+b=f a+c=j a+d=i  a+g=j a+h=i
b+c=g b+d=i b+e=i b+h=i
c+d=h c+e=i c+f=j  d+e=i d+f=i d+g=i d+j=i
e+f=i e+g=i e+h=i e+j=i
f+g=j f+h=i  g+h=i  h+j=i
Result for A=Enriched 8-element fence:  |Sub(A)| = 312, whence
sigma(A) = |Sub(A)|*2^(8-|A|) =   78.0000000000000000 .

10-element snake; its edges are
 oa ob ac ad bd ce de df eg eh fh gi hi
|A|=10, A(without commas)={abcdefghio}. Constraints:
a+b=d  b+c=e  c+d=e c+f=h  e+f=h  f+g=i  g+h=i
Result for A=10-element snake:  |Sub(A)| = 502, whence
sigma(A) = |Sub(A)|*2^(8-|A|) =  125.5000000000000000 .

The computation took 78/1000 seconds.
\end{verbatim}
\normalfont

\end{document}